\documentclass{amsart}

\makeatletter
\def\subsection{\@startsection{subsection}{2}
  \z@{.5\linespacing\@plus.7\linespacing}{.5\linespacing}
  {\normalfont\bfseries}}
\makeatother

\makeatletter
\def\@defaultbiblabelstyle#1{[#1]}
\makeatother

\makeatletter
\def\@setauthors{
  \begingroup
  \def\thanks{\protect\thanks@warning}
  \trivlist
  \centering\footnotesize \@topsep30\p@\relax
  \advance\@topsep by -\baselineskip
  \item\relax
  \author@andify\authors
  \def\\{\protect\linebreak}
  \authors
  \ifx\@empty\contribs
  \else
    ,\penalty-3 \space \@setcontribs
    \@closetoccontribs
  \fi
  \endtrivlist
  \endgroup
}
\def\@settitle{\begin{center}
  \baselineskip14\p@\relax
    \bfseries
  \@title
  \end{center}
}
\makeatother

\usepackage{amssymb}
\usepackage{amsxtra}
\usepackage{mathtools}
\usepackage[colorlinks=true, citecolor=blue]{hyperref}
\usepackage{upref}
\usepackage{soul}

\newtheorem{theorem}{Theorem}[section]
\newtheorem{lemma}[theorem]{Lemma}
\newtheorem{proposition}[theorem]{Proposition}

\newtheorem{conjecture}[theorem]{Conjecture}
\newtheorem{question}[theorem]{Question}
\newtheorem{problem}[theorem]{Problem}

\theoremstyle{definition}
\newtheorem{definition}[theorem]{Definition}
\newtheorem{example}[theorem]{Example}

\theoremstyle{remark}

\numberwithin{equation}{section}
\allowdisplaybreaks[4]

\begin{document}

\title{Complete Simplicial Fans, Stanley--Reisner Rings, and equivariant h-polynomial}

\author{Tao Gui}
\address{(Tao Gui) \newline \indent Beijing International Center for Mathematical Research, Peking University, No.\ 5 Yiheyuan Rd, Haidian District, Beijing 100871, China}
\email{guitao18(at)mails(dot)ucas(dot)ac(dot)cn}



\begin{abstract}
We derive a graded character formula for the action of any finite group on the Artinian reduction of the Stanley--Reisner ring of any complete simplicial fan, which is given by an equivariant version of the classical h-polynomial. This gives the graded character formula for the representation of the group on the cohomology of the associated toric variety when the fan is rational. 
As an application, we use a navel tool, which we called hybrid fan, to compute the Poincar\'e polynomial of the invariants of the Artinian reduction of the Stanley--Reisner ring of any complete simplicial fan under a finite reflection group action. This implies that the Poincar\'e polynomial of the quotient of a compact toric orbifold by any finite reflection group is equal to the Poincar\'e polynomial of the compact toric orbifold associated to the hybrid fan.
\end{abstract}

\maketitle

\setcounter{tocdepth}{2}

\section{Introduction}
A fundamental theorem in toric geometry states that there is a one-to-one correspondence between (normal) toric varieties and fans of strongly convex rational (with respect to some lattice) polyhedral cones\cite{davis1991convex}. This correspondence gives a rich dictionary between the geometry of a toric variety and the combinatorics of its associated fan. For example, the toric variety is an orbifold (that is, it has only finite quotient singularity) if and only if the associated fan is simplicial, and the toric variety is compact (in the classical topology) if and only if the associated fan is complete\cite[Theorem 3.1.19]{cox2024toric}. Furthermore, the toric variety is projective if and only if the associated fan is polytopal (that is, it is the normal fan of some rational convex polytope)\cite[Theorem 6.2.1]{cox2024toric}. 

If a finite group $W$ acts on a polytope or a fan preserving the defining lattice, then this action induces an action of $W$ on the associated toric variety, and hence induces a graded representation of $W$ on its cohomology. For the case of a Weyl group $W$ acting on the standard $W$-permutohedron and the normal fan consisting of Weyl chambers, this representation has been extensively studied by Procesi\cite{Procesi90}, by Dolgachev--Lunts\cite{dolgachev1994character}, by Stembridge\cite{stembridge1994some}, and by Lehrer\cite{lehrer2008rational}. Each of these papers provides a different perspective.

However, polytopes and fans can be non-rational with respect to any lattice\footnote{In fact, there exist combinatorial types of polytopes which do not come from any rational polytope, see for example, \cite{ziegler2008nonrational}.}. In such cases, the associated toric varieties does not exist. Nevertheless, there is always a finite poset, the \emph{face poset} of the polytope or fan, which encodes the torus orbit stratification when there is a toric variety, known as the orbit-cone correspondence in toric geometry\cite[Theorem 3.2.6]{cox2024toric}. If a group acts on a polytope or fan, then it acts on the face poset preserving the partial order given by inclusion of faces.

Note that the Betti numbers and hence the usual cohomology of a general toric variety are not invariants of the combinatorics of the associated fan in general\cite{mcconnell1989rational}. However, for an important class of toric variaties---the compact toric orbifolds associated with \emph{rational complete simplicial fans}, the usual cohomology (at least as a graded vector space) of the toric variety depends only on the face poset of the fan\cite[Theorem 12.4.1 and Corollary 12.4.8]{cox2024toric}. This fact can be derived from the localization theorem of equivariant cohomology and Goresky--Kottwitz--MacPherson (GKM) theory\cite{goresky1998equivariant}, where an important ingredient from geometry is that the torus action on this class of toric variety is \emph{equivariantly formal}. Actually, one can define a graded ring $\mathcal{A}(\Sigma)$, known as the \emph{Stanley--Reisner ring}, for any complete simplicial fan $\Sigma$, see Section \ref{sec-SR}. It is an algebra over a polynomial ring $\mathcal{A}$, and hence in particular
a graded $\mathcal{A}$-module. If $\Sigma$ is rational with respect to some lattice, the ring $\mathcal{A}(\Sigma)$ is actually canonically isomorphic to the torus-equivariant cohomology of the associated toric variety, the ring $\mathcal{A}$ is isomorphic to the torus-equivariant cohomology of a point, and equivariantly formality implies that $\mathcal{A}(\Sigma)$ is a \emph{free} $\mathcal{A}$-module. It follows that the usual cohomology of the associated toric variety is isomorphic to the Artinian reduction $\overline{\mathcal{A}(\Sigma)}$---the quotient of $\mathcal{A}(\Sigma)$ by the augmentation ideal of $\mathcal{A}$. Then the Poincar\'e polynomial of the associated toric variety can be computed using the combinatorics of the fan, known as the \emph{h-polynomial}. However, the definitions of $\mathcal{A}(\Sigma)$ and $\mathcal{A}$ does not depend on the existence of the toric variety and the free-ness holds for any complete simplicial fan $\Sigma$, see \cite{stembridge1994some}.

As a consequence, when there is a finite group G acting linearly on a vector space $V$ preserving a polytope or more generally a complete simplicial fan $\Sigma$ in $V$, one can study the induced action on the face poset, and the induced graded representation on $\overline{\mathcal{A}(\Sigma)}$, which may give some useful combinatorial information on the polytope or the fan, for example, information about the face numbers. In this paper, we study this representation and give the graded character, which we call the \emph{equivariant h-polynomial}. Denote the graded character of $g \in G$ on ${\overline{\mathcal{A}(\Sigma)}}$ by 
\begin{equation} \label{eq-def of graded cha}
P_G(t, g):=\sum_{i \geqslant 0} \operatorname{Trace}\left(g, \overline{\mathcal{A}(\Sigma)}^{i}\right) t^i \in \mathbb{C}[t],    
\end{equation}
whose information is equivalent to the equivariant Hilbert series of $\overline{\mathcal{A}(\Sigma)}$
$$
\operatorname{R}_{\overline{\mathcal{A}(\Sigma)}}(t):=\sum_i [\overline{\mathcal{A}(\Sigma)}^{i}] t^i \in R(G)[t],
$$
where $R(G)$ denotes the complex character ring of $G$. For any $\sigma \in \Sigma$, let $G_{\sigma }$ denote the setwise stabilizer of $\sigma \in \Sigma$, let $V_\sigma$ be the linear span of the cone $\sigma$ in $V$, then we have an induced quotient representation $\psi_{\sigma}: G_{\sigma} \rightarrow \mathrm{GL}\left(V/ V_\sigma\right)$. 
The following is one of our main results.

\begin{theorem} \label{thm-main}
For a finite group $G$ acting linearly on $V$ preserving a complete simplicial fan $\Sigma$, the equivariant Hilbert series of $\overline{\mathcal{A}(\Sigma)}$ is given by\footnote{Although the expression on the right hand side of \eqref{eq-general} contains fractions, an expression involving only integers can be obtained by allowing $\sigma$ to vary over a choice of representatives for the $G$-orbits of $\Sigma$.}
\begin{equation} \label{eq-general}
   \operatorname{R}_{\overline{\mathcal{A}(\Sigma)}}(t)=\sum_{\sigma \in \Sigma} \frac{\left|G_{\sigma }\right|}{\left|G\right|} \operatorname{Ind}_{G_{\sigma}}^{G}\left(\operatorname{det}(tI-\psi_\sigma)\right),
\end{equation}
where $\operatorname{det}(tI-\psi_\sigma)$ is viewed as a $\mathbb{C}[t]$-valued class function on $G_{\sigma }$ whose value is $\operatorname{det}(tI-\psi_\sigma(g))$ when evaluates at $g \in G$.
\end{theorem}

It is not hard to see that the case $g=\operatorname{id}$ of our formula \eqref{eq-general} gives the usual h-polynomial \eqref{eq-h-poly} of $\Sigma$, hence we call the right hand side of \eqref{eq-general} the \emph{equivariant h-polynomial} of $\Sigma$ under the action of $G$. As mentioned earlier, \eqref{eq-general} gives the equivariant Hilbert series of the cohomology $H^*\left(X_{\Sigma}; \mathbb{R}\right)$ of a compact toric orbifold $X_{\Sigma}$ under a finite group $G$ when the $G$  action preserves a lattice $M \subset V$ and an $M$-rational complete simplicial fan $\Sigma$ in $V$. 

In \cite{stapledon2011equivariant}, Stapledon introduce and develop an equivariant generalization of the classical Ehrhart theory by studying lattice points and representations in dilations of an invariant lattice polytope under a finite group $G$ action. He prove in \cite[Proposition 8.1]{stapledon2011equivariant} that under certain conditions, his \emph{equivariant $h^*$-polynomial} is equal to the equivariant Hilbert series of the cohomology of toric variety associated to a smooth $G$-invariant rational fan $\Delta$. Therefore, under the conditions of \cite[Proposition 8.1]{stapledon2011equivariant}, assume further that $\Delta$ is complete, then the corresponding equivariant $h^*$-polynomial in \cite[Proposition 8.1]{stapledon2011equivariant} can be computed by the equivariant h-polynomial of $\Delta$ using the right hand side of in \eqref{eq-general}.

When $G$ is a Weyl group $W$ acting on the standard $W$-permutohedron and the normal fan consisting of Weyl chambers, our formula \eqref{eq-general} recover the well known formula due to Procesi\cite{Procesi90}, Dolgachev--Lunts\cite{dolgachev1994character}, Stembridge\cite{stembridge1994some}, and Lehrer\cite{lehrer2008rational}; see \cite[Theorem 1.1]{lehrer2008rational}. In fact, Stembridge\cite[Theorem 1.4]{stembridge1994some} derived a graded character formula for a finite subgroup $G$ of the orthogonal group $O(V)$ acting on a general complete simplicial fan $\Sigma$, with the additional assumption that the action of $G$ on $\Sigma$ is \emph{proper}. By definition, an action of $G$ on $\Sigma$ is proper if for any $g \in G$, $g$ fixes a cone $\sigma$ setwise, then it must fix $\sigma$ pointwise. By a similar argument as in \cite[Appendix A]{lehrer2008rational} combined with our proof of \eqref{eq-general}, our formula \eqref{eq-general} recover Stembridge's formula under his assumption of proper actions.

Some work has been done to study combinatorics of simplicial polytopes with certain simple symmetries, see, e.g., \cite{stanley1987number,adin1995face,jorge2003combinatorics}. It is interesting to see whether our graded character formula \eqref{eq-general} can be investigated to study the combinatorics of simplicial polytopes and complete simplicial fans with symmetries.

An important class of symmetries of polytopes and fans consists of the \emph{reflection symmetries}. For example, it is well-known that the groups of all symmetries of Platonic solids are finite reflection groups (that is, a finite group which is generated by reflections in an Euclidean space)\cite{coxeter1973regular}. More generally, a class of polytopes with many reflection symmetries is the class of $W$-permutohedrons (or weight polytopes). By definition, a $W$-permutohedron is the convex hull of a reflection group $W$-orbit of a point in the weight space and thus $W$ has a natural action on it. 

When $W$ is a finite reflection group acting on a complete simplicial fan $\Sigma$, we have the following consequences of the above general formula \eqref{eq-general}.

\begin{theorem} \label{thm-ref}
For a finite reflection group $W$ on $V$ preserving a complete simplicial fan $\Sigma$, then the Poincar\'e polynomial of the $W$-invariants 
$\mathrm{P}_{\overline{\mathcal{A}(\Sigma)}^W}(t)$ is equal to the h-polynomial of the hybrid fan $\Sigma_W$, defined in Definition \ref{def-hybrid fan}.
\end{theorem}

As mentioned earlier, compact toric orbifolds are given by complete simplicial fans which are rational with respect to some lattice\cite[Theorem 3.1.9]{cox2024toric}, and the cohomology of the compact toric orbifold is given by the Artinian reduction of the Stanley--Reisner ring of the corresponding fan. Since it is well known that the cohomology of the quotient of a variety under a finite group $G$ is isomorphic to the $G$-invariant of the cohomology of the original variety, a consequence of Theorem \ref{thm-ref} is the following 

\begin{theorem} \label{thm-Poincare}
Let $W$ be a finite reflection group on $V$ preserving a lattice $M \subset V$ and an $M$-rational complete simplicial fan $\Sigma$ in $V$, with associated compact toric orbifold $X_\Sigma$. Then the Poincar\'e polynomial of the cohomology $H^*\left(X_\Sigma / W ; \mathbb{R}\right)$ of the quotient $X_\Sigma / W$ equals the Poincar\'e polynomial of the cohomology $H^*\left(X_{\Sigma_W}; \mathbb{R}\right)$ of the compact toric orbifold $X_{\Sigma_W}$ associated to the hybrid fan $\Sigma_W$.
\end{theorem}

Note that under the condition of Theorem \ref{thm-Poincare}, by Proposition \ref{prop-completeness}, Proposition \ref{prop-simplicial}, and \cite[Proposition 2.5]{crowley2024toric}, the hybrid fan $\Sigma_W$, defined in Definition \ref{def-hybrid fan}, is still a $M$-rational complete simplicial fan hence it gives a compact toric orbifold $X_{\Sigma_W}$.

The motivation of Theorem \ref{thm-ref} and Theorem \ref{thm-Poincare} comes from the work of \cite{HMSS21-toric-orbifolds}. In \cite{HMSS21-toric-orbifolds}, Horiguchi--Masuda--Shareshian--Song considered toric variety associated with Weyl chambers (whose fan is the normal fan $\Sigma_{P^{\lambda}}$ of a weight polytope $P^\lambda$ associated with a regular weight $\lambda$) of all classical Lie types and proved the invariants $H^*\left(X_{\Sigma_{P^{\lambda}}}\right)^{W_J}$ of the rational cohomology of the toric variety $X_{\Sigma_{P^{\lambda}}}$ under any parabolic subgroup $W_J$ of the Weyl group is isomorphic to the rational cohomology $H^*\left(X_{\Sigma_{(P^{\lambda}/W_J)}}\right)$ of the toric varieties associated with the normal fan of the quotient polytope $P^{\lambda}/W_J$. Based on this result, they asked in \cite[Question 8.3]{HMSS21-toric-orbifolds} whether the quotient $X_{\Sigma_P}/W$ of the projective toric variety $X_{\Sigma_P}$ associated to the normal fan $\Sigma_P$ of any simple polytope $P$ modulo an action of a finite reflection group $W$ is isomorphic to the projective toric variety $X_{\Sigma_{P/W}}$ associated with the normal fan of the quotient polytope $P/W$. Soon-after, Crowley--Gong--Simpson\cite{crowley2024toric} gave a positive answer to this question even without assuming the polytope $P$ being simple. During their study for \cite[Question 8.3]{HMSS21-toric-orbifolds} , Colin Crowley and Connor Simpson came up with the definition of the hybrid fan $\Sigma_W$. Since the normal fan $\Sigma_{P / W}$ of the any quotient polytope $P / W$ is the same as the hybrid fan $\left(\Sigma_P\right)_W$ of the normal fan $\Sigma_P$ of $P$ with the reflection group $W$ (see Proposition \ref{prop-polytopal}), the above Theorem \ref{thm-Poincare} may be regarded as a generalization (of course only on the Poincar\'e polynomial level) of the results in \cite{HMSS21-toric-orbifolds} to arbitrary (possibly non-rational) complete (possibly non-polytopal) simplicial fan, and is a shadow of isomorphisms on the cohomology level and variety level when the fan is rational. We will discuss some related open question in Section \ref{sec-ques}.

The rest of this paper is organized as follows. In Section \ref{sec-pre}, we review the Stanley--Reisner rings and h-polynomials of polytopes and fans, and introduce group actions on fans and Stanley--Reisner rings. In Section \ref{sec-proof}, we prove that the graded character $
P_G(t, g)\in \mathbb{C}[t]$ in \eqref{eq-def of graded cha} is palindromic as a polynomial in $t$ and gave a proof of Theorem \ref{thm-main}. In Section \ref{sec-ref}, we study reflection actions on fans and Stanley--Reisner rings, and gave a proof of Theorem \ref{thm-ref}. Finally, in Section \ref{sec-ques}, we present some open questions and conjectures for future research.

Although our setting is  entirely in convex geometry,
in many places we point out the connections with the topology of toric varieties.
Although this is not necessary to understand our proofs, these connections are essential for us to motivate the
statements and the topological interpretation  provides a powerful way of understanding what these results
mean.

\subsection*{Acknowledgments}
The author is grateful to Colin Crowley and Connor Simpson for kindly sharing their definition of the hybrid fan by personal communication and to Connor Simpson for other valuable discussions, and would like to thank Tom Braden, Feifei Fan, Tao Gong, Hongsheng Hu, Nick Proudfoot, Alan Stapledon, Matthew Hongye Xie, and Haozhi Zeng for useful communications. This research was done while the author was
visiting the Institute of Mathematical Sciences in ShanghaiTech University and the Institute for Math \& AI in Wuhan University, the author thank Junbin Dong and Chenglang Yang for their hospitality. The author is supported in part by NSFC: 12471309.

\section{Preliminaries} \label{sec-pre}

In this section, we fix our notations and definitions for later use.

\subsection{Polytopes and Fans}

Let $V$ be a $d$-dimensional Euclidean space with a positive definite bilinear form $\langle-,-\rangle$, let $P$ denote a $d$-dimensional polytope in $V$. 

A (convex polyhedral) \emph{cone} in $V$ is a set of the form $\sum \mathbb{R}_{\geq 0} v_i$ with $v_i \in V, i=1, \cdots, n$, it is \emph{strongly convex} if it does not contain any vector subspace other than the origin. A \emph{fan} $\Sigma$ in $V$ is a finite collection of strongly convex polyhedral cones in $V$ so that every face of a strongly convex cone
in $\Sigma$ is again in $\Sigma$, and the intersection of any two cones in $\Sigma$ is a face of each. We write $\Sigma_k$ to denote the set of all cones in $\Sigma$ of dimension $k$. A one-dimensional cone in $\Sigma_1$ is called a \emph{ray}. Note that the set of all faces of a cone $\sigma$ forms a fan $[\sigma]$, and we denotes $\sigma(1)$ to be the set of rays of this fan.

A polytope $P$ in $V$ gives rise to two fans $\Delta_P$ and $\Sigma_P$ in $V$. The first fan $\Delta_P$ is known as the \emph{central fan} of $P$, by choosing the origin to be an interior point of $P$ and coning off all faces $F \neq P$. The second fan $\Sigma_P$ is known as the \emph{normal fan} of $P$, which can be defined to be the central fan of the polar dual $P^*$ of P. 

One can get a poset $\mathcal{F}(\Sigma)$ out of a fan $\Sigma$, called the \emph{face poset} or \emph{combinatorial type} of $\Sigma$, where the elements in $\mathcal{F}(\Sigma)$ are the cones in $\Sigma$ and the partial order $\tau \leq \sigma$ means a cone $\tau$ is a face of $\sigma$. 

The \emph{support} $|\Sigma|$ of $\Sigma$ is the union of all its cones. A fan $\Sigma$ in $V$ is \emph{complete} if $|\Sigma|=V$. The two fans $\Delta_P$ and $\Sigma_P$ from a polytope are complete, but there are combinatorial types of a complete fan which can not be realized as that of the normal fan of any convex polytope, see for example \cite[Example 5.10]{ziegler2012lectures}. 

A fan $\Sigma$ in $V$ is \emph{simplicial} if each cone in $\Sigma$ is a cone over a
simplex. Thus the fan $\Delta_P$ is a simplicial fan if and only if $P$ is a simplicial polytope, which means each face of $P$ is a simplex. 

\subsection{Stanley--Reisner rings and h-polynomials} \label{sec-SR}
Fix a complete simplicial fan $\Sigma$ in $V$. Let $\mathcal{A}=\operatorname{Sym}\left(V^*\right)$ be the ring of polynomial functions on $V$. 
A function $f:V \rightarrow \mathbb{R}$ is called \emph{conewise polynomial} if for all $\sigma \in \Sigma$ the restriction $\left.f\right|_\sigma$ is a polynomial. The set $\mathcal{A}(\Sigma)$ of such functions is a commutative graded (with the usual grading where linear functions have degree 1) ring under the operations of pointwise addition and multiplication. There is a natural injective ring homomorphism $\mathcal{A} \rightarrow \mathcal{A}(\Sigma)$ obtained by restricting polynomials to $|\Sigma|$. This makes $\mathcal{A}(\Sigma)$ into an algebra over $\mathcal{A}$, and in particular a graded $\mathcal{A}$-module.

After a moment thought, one can see that $\mathcal{A}(\Sigma)$ is isomorphic to the quotient of the polynomial ring $\mathbb{R}\left[x_\rho\right]$ with one generator for each ray $\rho \in \Sigma_1$ by the square-free monomial ideal
$$
\left.\left\langle x_{\rho_1} \cdots x_{\rho_k}\right| \rho_1, \ldots, \rho_k \text { are not the rays of a cone in } \Sigma\right\rangle.
$$
This ring is known as the \emph{Stanley--Reisner ring}, or \emph{face ring}, of $\Sigma$, when $\Sigma$  is viewed as an abstact simplicial complex. The isomorphism comes by identifying the generator $x_\rho$ with a conewise linear function which restricts to a one on a ray generator $v$ of $\rho$ and to zero on all other $\rho^{\prime} \in \Sigma_1$. 

Since every non-zero monomial $m$ in the Stanley--Reisner ring $\mathcal{A}(\Sigma)$ of $\Sigma$ has a minimal face $\sigma \in \Sigma$ supporting it (i.e., $m = \prod_{\rho_i \in \sigma(1)} x_i^{a_i}$ with all $a_i \ge 1$), as a grade vector space over $\mathbb{R}$, $\mathcal{A}(\Sigma)$ admits a direct sum decomposition
\begin{equation} \label{eq-decom of SR}
   \mathcal{A}(\Sigma)=\bigoplus_{\sigma \in \Sigma} \mathcal{A}(\Sigma)_\sigma, 
\end{equation}
where the summand $\mathcal{A}(\Sigma)_\sigma$ has a basis formed by the monomials (not necessarily square-free) whose support is exactly the rays of cone $\sigma$ in $\Sigma$, hence as graded vector space, we have
\begin{equation} \label{eq-summand of SR}
\mathcal{A}(\Sigma)_\sigma=x_\sigma \cdot \mathbb{C}[x_{\rho_i} \mid \rho_i \in \sigma(1)]
\end{equation}
where $x_\sigma = \prod_{\rho_i \in \sigma(1)} x_{\rho_i}$ is the square-free monomial of degree $\dim(\sigma)$.

Therefore, a simple counting on monomials in the $x_\rho$ shows that the Hilbert series of $\mathcal{A}(\Sigma)$ is
$$
\operatorname{Hilb}(\mathcal{A}(\Sigma), t)=h_\Sigma( t)/(1-t)^d,
$$
where 
\begin{equation} \label{eq-h-poly}
h_\Sigma( t):=\sum_i\left|\Sigma_i\right|(t-1)^{d-i}
\end{equation}
is called the \emph{h-polynomial} of $\Sigma$. It is easy to see that h-polynomial of $\Sigma$ contains the same information as in the face numbers $f_i(\Sigma):=\left|\Sigma_i\right|$ of $\Sigma$.

For any graded $\mathcal{A}$-module $M$, we denote $\overline{M}:=M / \mathfrak{m} M=M \otimes_\mathcal{A} \mathcal{A} / \mathfrak{m}$, the quotient by the maximal ideal $\mathfrak{m}$ generated by the global linear functions $\mathcal{A}_1$. Since any complete simplicial fan $\Sigma$ gives a triangulation of a sphere (essentially by definition), by Reisner's criterion \cite[Theorem 5.53]{miller2005combinatorial}, $\mathcal{A}(\Sigma)$ is \emph{Cohen--Macaulay}. Hence $\mathcal{A}(\Sigma)$ is a free $\mathcal{A}$-module for the complete simplicial fan $\Sigma$ \cite[Theorem 13.37]{miller2005combinatorial} and we have the $\mathcal{A}$-module isomorphism 
\begin{equation} \label{eq-free}
  \mathcal{A}(\Sigma)\cong \mathcal{A}\otimes_{\mathbb{R}}\overline{\mathcal{A}(\Sigma)}.  
\end{equation}

In the language of commutative algebra, $\overline{\mathcal{A}(\Sigma)}$ is an \emph{Artinian reduction} of $\mathcal{A}(\Sigma)$---the quotient of $\mathcal{A}(\Sigma)$ by a specific \emph{linear system of parameters} of degree 1. Let the rays of the fan $\Sigma$ be $\rho_1,\dots,\rho_n$ with a chosen ray generators $v_i \in V$. Let $V^*$ be the dual space of linear functionals on V. There is an natural injective linear map
\begin{equation}\label{eq-injective}
    \Psi: V^* \longrightarrow \mathcal{A}(\Sigma)_1, \qquad f \longmapsto \sum_{i=1}^n f(v_i)\,x_i .
\end{equation}
The image $U = \Psi(V^*)$ is a d-dimensional subspace of $\mathcal{A}(\Sigma)_1$. Pick a basis $\theta_1,\dots,\theta_d$ of $V^*$. Because the fan is complete and simplicial, $\Psi(\theta_1), \dots,\Psi(\theta_d)$ form a regular sequence for $\mathcal{A}(\Sigma)$ (cf. \cite[eq. (1.2), p.250]{stembridge1994some}, known as the toric linear system of parameters. It follows that 
$\overline{\mathcal{A}(\Sigma)}$ is isomorphic to the Artinian reduction \begin{equation} \label{eq-toric}
    \overline{\mathcal{A}(\Sigma)}\cong\mathcal{A}(\Sigma)/\left(\Psi(\theta_1), \dots,\Psi(\theta_d)\right).
\end{equation}

Note that $\overline{\mathcal{A}(\Sigma)}$ has a induced graded ring structure from $\mathcal{A}(\Sigma)$, hence it is an $\mathbb{R}$-graded algebra.
A basis of $\mathcal{A}(\Sigma)$ over $\mathcal{A}$ is obtained by lifting an $\mathbb{R}$-basis of $\overline{\mathcal{A}(\Sigma)}$. 
 Since the Hilbert series of $\mathcal{A}$ is $(1-t)^{-d}$, this gives
$$
\operatorname{Hilb}(\overline{\mathcal{A}(\Sigma)}, t)=h_\Sigma( t).
$$ In particular, the h-polynomial $h_\Sigma( t)$ of the complete simplicial fan $\Sigma$ is non-negative. 

This non-negativity can also be derived from toric geometry. By small deformations of the ray\footnote{This deformation may break the symmetries of the fan!}, a simplicial fan $\Sigma$ can be made rational with respect to a fixed lattice in $V$ without changing
its combinatorial type. So we may assume $\Sigma$ is rational when we are only interested in its Stanley--Reisner ring and h-polynomial. For a rational complete simplicial fan $\Sigma$, one can construct the associated toric variety $X_\Sigma$. Then the ring $\mathcal{A}(\Sigma)$ is actually isomorphic to the torus-equivariant cohomology of $X_\Sigma$, the ring $\mathcal{A}$ is  isomorphic to the torus-equivariant cohomology of a point, and the ring homomorphism $\mathcal{A} \rightarrow \mathcal{A}(\Sigma)$ is induce on the equivariant cohomology by a map from $X_\Sigma$ to a point. The freeness of $\mathcal{A}(\Sigma)$ as an $\mathcal{A}$-module comes from the equivariant formality of the torus action on $X_\Sigma$. One can show that $h_\Sigma( t)$ is actually equal to the Poincare polynomial of the usual cohomology of $X_\Sigma$ with real coefficients, which is isomorphic to $\overline{\mathcal{A}(\Sigma)}$. Since $\Sigma$ is a complete simplicial fan, $X_\Sigma$ is known to be a complete toric variety which is rationally smooth, hence the usual cohomology of $X_\Sigma$ with rational coefficients satisfies the Poincare duality. Passing to the Poincare polynomial level, the duality implies that $h_\Sigma( t)$ is palindromic for a complete simplicial fan $\Sigma$, a fact known as the \emph{Dehn--Sommerville equations}. A complete characterization of $h_\Sigma( t)$ for a complete simplicial fan $\Sigma$ (or more generally, for a simplicial sphere), including the unimodality of $h_\Sigma( t)$, is known as the celebrated \emph{g-conjecture}, formulated by McMullen in 1971 \cite{mcmullen1971numbers}. For a simplicial polytope $P$, the g-conjecture is proved in 1980 by Billera and Lee (sufficiency)\cite{billera1981proof} and Stanley (necessity)\cite{Stanley80}. The unimodality of $h_{\Sigma_P}( t)$ follows from the hard Lefschetz theorem for $\overline{\mathcal{A}(\Sigma)}$, which holds because $X_{\Sigma_P}$ is projective for a polytopal fan $\Sigma_P$. McMullen\cite{mcmullen1993simple} later gave a proof of the hard Lefschetz theorem for $\overline{\mathcal{A}(\Sigma_P)}$ which did not involve toric varieties and so worked for a non-rational polytope $P$. For a complete simplicial fan $\Sigma$ (or more generally, for a simplicial sphere), the difficulty to prove the g-conjecture is the lack of ``ample'' elements when one tries to establish the hard Lefschetz theorem for $\overline{\mathcal{A}(\Sigma)}$, see for example, \cite{adiprasito2018combinatorial} and \cite{karu2023anisotropy}.

\subsection{Group actions on fans and Stanley--Reisner rings}

Now suppose $\Sigma$ is a fan in $V$ which is invariant under the linear action of a finite group $G$ on $V$. Hence the action of $G$ maps cones to cones and induces an action of $G$ on the face poset $\mathcal{F}(\Sigma)$. For any $\sigma \in \mathcal{F}(\Sigma)$, let $G_\sigma=\{g \in G: g \cdot \sigma=\sigma\}$ be the stabilizer of $\sigma$, which is also the setwise stabilizer of $\sigma$ as a cone under the action of $g$ on $V$.  An important class of actions are defined in \cite{stembridge1994some}: the action of $G$ on $\Sigma$ is \emph{proper} if for any $g \in G$, $g$ fixes a cone $\sigma$ setwise, then it must fix $\sigma$ pointwise. In terms of the face poset, this condition implies that for every $\sigma \in \mathcal{F}(\Sigma)$ and $g \in G$, $g\cdot\sigma=\sigma$, then $g\cdot\tau=\tau$ for $\tau\leq\sigma$.
For any $\sigma \in \mathcal{F}(\Sigma)$, let $V_\sigma$ be the linear span of the cone $\sigma$ in $V$, then we have an induced quotient representation $\psi_{\sigma}: G_{\sigma} \rightarrow \mathrm{GL}\left(V/ V_\sigma\right)$. 

Suppose further that the $G$-invariant fan $\Sigma$ is complete and simplicial. Then it is easy to see that the $G$-actions on $V$ and $\mathcal{F}(\Sigma)$ induce a graded $G$-representation on the Stanley--Reisner ring $\mathcal{A}(\Sigma)$. It is not hard to see that the span of the toric linear system of parameters in \eqref{eq-toric} is $G$-invariant, see for example, \cite[p.250]{stembridge1994some}. Hence the Artinian reduction $\overline{\mathcal{A}(\Sigma)}$ is a graded representation of $G$. The isomorphism in \eqref{eq-free} becomes an isomorphism of graded $G$-modules (see \cite[p.251, eq. (2.4)]{stembridge1994some}):
\begin{equation} \label{eq-Giso}
  \mathcal{A}(\Sigma)\cong \mathcal{A}\otimes_{\mathbb{R}}\overline{\mathcal{A}(\Sigma)}.  
\end{equation}

\section{Proof of Theorem \ref{thm-main}} \label{sec-proof}
We need the following well-known lemma, see for example, \cite[Chapter 5, Section 5, No.3, Lemma 2]{Bourbaki-Lie456}.

\begin{lemma} \label{lem-sym}
If $V$ is a $n$-dimensional vector space and $g \in \text{End}(V)$ is a linear transformation, then graded character of $g$ on the symmetric algebra $Sym(V)$ is given by 

$$\sum_{k=0}^{\infty} \text{Tr}(g|_{Sym^k(V)}) t^k = \frac{1}{\det(I - tg)}.$$
\end{lemma}

We have the following equivariant version of the Dehn--Sommerville equations, which states that the graded character $
P_G(t, g)\in \mathbb{C}[t]$ of any $g \in G$ on ${\overline{\mathcal{A}(\Sigma)}}$ is palindromic. 

\begin{proposition}
Suppose a finite group $G$ acts linearly on $V$ preserving a complete simplicial fan $\Sigma$. Then for any $g \in G$, the graded character $
P_G(t, g)\in \mathbb{C}[t]$, defined in \eqref{eq-def of graded cha}, satisfies 
\begin{equation} \label{eq-PD}
 P_G(t, g)=t^d P_G(t^{-1}, g)   
\end{equation}
where $d=\operatorname{deg}P_G(t, g)$ equals to the dimension of $V$.
\end{proposition}

\begin{proof}
    Because $\Sigma$ is a complete simplicial fan, the Stanley--Reisner ring $\mathcal{A}(\Sigma)$ is a Gorenstein ring \cite[Exercise 13.12]{miller2005combinatorial}. Consequently, the Artinian reduction $\overline{\mathcal{A}(\Sigma)}$ is a Poincaré duality algebra with a 1-dimensional socle in top degree $d$, see for example, \cite{brion1997structure}. That is, there exists a non-degenerate pairing for any $0\leq k \leq d$
\begin{equation} \label{eq-PP}
\overline{\mathcal{A}(\Sigma)}^k \times \overline{\mathcal{A}(\Sigma)}^{d-k} \to \overline{\mathcal{A}(\Sigma)}^d \cong \mathbb{R}.
\end{equation}
The above pairing is induced by multiplication in $\overline{\mathcal{A}(\Sigma)}$. When $G$ acts linearly on $V$ preserving the fan $\Sigma$, the non-degenerate pairing \eqref{eq-PD} is preserved by the $G$-action. 

It remains to prove that $G$ acts on $\overline{\mathcal{A}(\Sigma)}^d$ always trivially. Thus, the complementary degree pieces of the graded representation $\overline{\mathcal{A}(\Sigma)}$ are dual to each other. Since all the graded pieces are real representations, the complementary degree pieces are isomorphic to each other and the graded character $
P_G(t, g)$ for any $g \in G$ is palindromic.  

We prove $G$ acts on $\overline{\mathcal{A}(\Sigma)}^d$ by tracking the linear relations between adjacent maximal cones to build a globally $G$-invariant volume element. 

Because $G$ is a finite group acting linearly and preserving the fan, it permutes the rays. By choosing a $G$-invariant inner product and pick a unit vector $v_\rho$ for one ray $\rho$, this ensures $G$ permutes the ray generators $\{v_\rho\}$.

For any maximal cone $\sigma \in \Sigma_d$, let $x_\sigma = \prod_{\rho \in \sigma(1)} x_\rho$ be the square-free monomial of degree $d$. The top degree $\overline{\mathcal{A}(\Sigma)}^d$ is spanned by these $x_\sigma$ elements, subject to the relations in \eqref{eq-toric}.

Let $\sigma_1$ and $\sigma_2$ be two adjacent maximal cones sharing a facet $\tau$.
Let $\tau$ be generated by the rays $\{\rho_1, \dots, \rho_{d-1}\}$.
Let the remaining ray of $\sigma_1$ be $\rho_d$, and the remaining ray of $\sigma_2$ be $\rho'_{d}$.

Choose a non-zero linear functional $\alpha \in V^*$ that annihilates the hyperplane spanned by $\tau$. Thus, $\alpha(v_{\rho_1}) = \dots = \alpha(v_{\rho_{d-1}}) = 0$.
The linear system of parameters element for $\alpha$ must evaluate to $0$ in $\overline{\mathcal{A}(\Sigma)}$:
$$ \sum_{\rho \in \Sigma(1)} \alpha(v_\rho) x_\rho = 0.$$

Multiply this relation by the degree $(d-1)$ monomial $x_\tau = x_1 \cdots x_{d-1}$.
Because of the Stanley–Reisner ideal, any product of rays that do not form a cone in $\Sigma$ is $0$. The only rays that can form a cone with $\tau$ are $\rho_d$ and $\rho'_{d}$. Therefore, all terms in the sum vanish except two:
\begin{equation}\label{eq-adj}
 \alpha(v_{\rho_{d}}) x_{\sigma_1} + \alpha(v_{\rho'_{d}}) x_{\sigma_2} = 0.   
\end{equation}

Let us rewrite the coefficients $\alpha(v_{\rho_{d}})$ and $\alpha(v_{\rho'_{d}}) $ using determinants.
We can identify $\alpha(v)$ with the determinant form $\det(v_{\rho_1}, \dots, v_{\rho_{d-1}}, v)$ (up to a global non-zero scalar).
Therefore:
\begin{equation*}
 \begin{aligned}
  \alpha(v_{\rho_{d}}) = \det(v_{\rho_1}, \dots, v_{\rho_{d-1}}, v_{\rho_{d}}) = \det(v_{\sigma_1}),
\\
\alpha(v_{\rho'_{d}}) = \det(v_{\rho_1}, \dots, v_{\rho_{d-1}}, v_{\rho'_{d}}) = \det(v_{\sigma_2}).
 \end{aligned}   
\end{equation*}

Because $\Sigma$ is a fan, the vectors $v_{\rho_{d}}$ and $v_{\rho'_{d}}$ must lie on opposite sides of the hyperplane spanned by $\tau$. This means their determinants must have opposite signs:
$$ \operatorname{sgn}(\det(v_{\sigma_1})) = - \operatorname{sgn}(\det(v_{\sigma_2})). $$

Substituting the determinants into our adjacency relation \eqref{eq-adj} gives:
$$ \det(v_{\sigma_1}) x_{\sigma_1} + \det(v_{\sigma_2})x_{\sigma_2} = 0. $$

Dividing by the determinants and absorbing the negative sign via the absolute value, we get:
$$ \frac{x_{\sigma_1}}{|\det(v_{\sigma_1})|} = \frac{x_{\sigma_2}}{|\det(v_{\sigma_2})|}. $$

Since the fan $\Sigma$ is complete, any two maximal cones are connected by a sequence of adjacent facets. This implies that the element
$$ \omega = \frac{x_\sigma}{|\det(v_\sigma)|} $$
is globally defined in $\overline{\mathcal{A}(\Sigma)}^d$, independent of the choice of maximal cone $\sigma$. Since it is non-zero and $\overline{\mathcal{A}(\Sigma)}^d$ is 1-dimensional, $\omega$ forms a canonical basis for $\overline{\mathcal{A}(\Sigma)}^d$.

We now test the action of an arbitrary element $g \in G$ on our basis element $\omega$:
$$ g \cdot \omega = g \cdot \left( \frac{x_\sigma}{|\det(v_\sigma)|} \right) = \frac{x_{g\sigma}}{|\det(v_\sigma)|}.$$

Because $g$ is a linear operator acting on the vectors generating $\sigma$ and $g$ permutes the ray generators $v_{\rho}$ by our choice, the determinants satisfies:
$$ \det(v_{g\sigma}) = \det(g) \det(v_\sigma). $$

Taking the absolute value on both sides:
$$ |\det(u_{g\sigma})| = |\det(g)| |\det(u_\sigma)| $$

Since $G$ is a finite group acting on a real vector space, the eigenvalues of $g$ are roots of unity, meaning its determinant must be $\pm 1$. Therefore, we have
$$ |\det(v_{g\sigma})| = |\det(v_\sigma)|.$$

This implies
$$ g \cdot \omega = \frac{x_{g\sigma}}{|\det(v_{g\sigma})|} = \omega. $$

Because $g \cdot \omega = \omega$ for all $g \in G$, the group acts trivially on the basis element $\omega$. Thus, the action of $G$ on the top degree of the Artinian reduction is identically the trivial representation. This completes the proof of the proposition.
\end{proof}

We are now in a position to give the
\begin{proof}[Proof of Theorem \ref{thm-main}]
Our strategy relies on three main steps: decomposing the Stanley--Reisner ring $\mathcal{A}(\Sigma)$ into $G$-invariant pieces, and then passing to the augmented quotient $\overline{\mathcal{A}(\Sigma)}$ using \eqref{eq-Giso}, 
and applying equivariant Poincaré duality (the algebraic manifestation of the Dehn-Sommerville equations) to align the polynomial degrees.

We want to evaluate the graded character (trace) of $g \in G$. Since $G$ permutes the cones in $\Sigma$, $g$ maps the summand $\mathcal{A}(\Sigma)_\sigma$ in \eqref{eq-decom of SR} to $\mathcal{A}(\Sigma)_{g\sigma}$. Thus, the only non-zero trace contributions come from the set of cones fixed setwisely by $g$, denoted $\Sigma^g$. 

For a invariant cone $\sigma \in \Sigma^g$, $g$ permutes its extreme rays. Because $\Sigma$ is a simplicial fan, the rays of $\sigma$ are linearly independent and form a basis for $V_\sigma$. Therefore, the permutation representation of $g$ on the variables $\{x_i \mid \rho_i \in \sigma(1)\}$ is isomorphic to the linear action of $g$ on $V_\sigma$. 

The generator $x_\sigma$ in \eqref{eq-summand of SR} is fixed by $g$ (since permutations leave the product of all variables invariant) and contributes $t^{\dim \sigma}$. The symmetric algebra $\mathbb{C}[x_{\rho_i} \mid \rho_i \in \sigma(1)]$ in \eqref{eq-summand of SR} contributes $1 / \operatorname{det}(I - t g_{V_\sigma})$ by Lemma \ref{lem-sym}. Summing over all invariant cones yields the graded character of $g \in G$ for the full Stanley-Reisner ring:
$$\operatorname{tr}(g \mid \mathcal{A}(\Sigma), t) = \sum_{\sigma \in \Sigma^g} \frac{t^{\dim \sigma}}{\operatorname{det}(I - t g_{V_\sigma})}.$$

Consider the isomorphism of graded $G$-modules in \eqref{eq-Giso}:
$$\mathcal{A}(\Sigma)\cong \mathcal{A}\otimes_{\mathbb{R}}\overline{\mathcal{A}(\Sigma)}.$$
Since the base field is $\mathbb{R}$, by Lemma \ref{lem-sym}, the graded character of $\mathcal{A}\cong\operatorname{Sym}(V^*)$ is $1 / \operatorname{det}(I - t g_{V^*}) = 1 / \operatorname{det}(I - t g_V)$  . Taking graded characters on both sides, we get the graded character of the Artinian reduction:
\begin{equation} \label{eq-cha from Giso}
  P_G(t, g) = \operatorname{det}(I - t g_V) \sum_{\sigma \in \Sigma^g} \frac{t^{\dim \sigma}}{\operatorname{det}(I - t g_{V_\sigma})}. 
\end{equation}

By Maschke's theorem, $V \cong V_\sigma \oplus (V/V_\sigma)$ as $g$-modules, which implies $\operatorname{det}(I - t g_V) = \operatorname{det}(I - t g_{V_\sigma}) \operatorname{det}(I - t g_{V/V_\sigma})$. This simplifies \eqref{eq-cha from Giso} to
\begin{equation} \label{eq-cha over inva cones}
P_G(t, g) = \sum_{\sigma \in \Sigma^g} t^{\dim \sigma} \operatorname{det}(I - t g_{V/V_\sigma}).
\end{equation} 

For an operator $A$ of dimension $k$, standard linear algebra dictates that $t^k \operatorname{det}(I - t^{-1}A) = \operatorname{det}(tI - A)$. Let's apply the palindromic property \eqref{eq-PD} to our sum \eqref{eq-cha over inva cones}:
\begin{equation} \label{eq-reverse}
    \begin{aligned}
P_G(t, g)=t^d P_G(t^{-1}, g) &= t^d \sum_{\sigma \in \Sigma^g} t^{-\dim \sigma} \operatorname{det}(I - t^{-1} g_{V/V_\sigma}) \\
&= \sum_{\sigma \in \Sigma^g} t^{d - \dim \sigma} \operatorname{det}(I - t^{-1} g_{V/V_\sigma}) \\
&= \sum_{\sigma \in \Sigma^g} \operatorname{det}(t I - g_{V/V_\sigma}),
\end{aligned} 
\end{equation}
where the last equality holds because $g_{V/V_\sigma}$ operates on a space of dimension $d - \dim(\sigma)$. 

We have established the character evaluated at a specific $g \in G$ as a sum over its fixed cones. The target formula \eqref{eq-general} is expressed globally using induced representations. Let's evaluate the target formula at $g$:
\begin{equation} \label{eq-eva at g}
\operatorname{R}_{\overline{\mathcal{A}(\Sigma)}}(t)(g) = \sum_{\sigma \in \Sigma} \frac{|G_\sigma|}{|G|} \operatorname{Ind}_{G_\sigma}^G \left(\operatorname{det}(tI-\psi_\sigma)\right)(g).    
\end{equation}  

Recall the Frobenius's formula for a general induced character: $$\operatorname{Ind}_{H}^G(\chi)(g) = \frac{1}{|H|} \sum_{h \in G, \ h^{-1}gh \in H} \chi(h^{-1}gh),$$
where $\chi$ is any character of a subgroup $H\subset G$.
Substituting this into the sum \eqref{eq-eva at g}, we have
\begin{equation} \label{eq-double}
\operatorname{R}_{\overline{\mathcal{A}(\Sigma)}}(t)(g) = \frac{1}{|G|} \sum_{\sigma \in \Sigma} \sum_{\substack{h \in G \\ h^{-1}gh \in G_\sigma}} \operatorname{det}(tI - \psi_\sigma(h^{-1}gh)_{V/V_\sigma})   
\end{equation}

Let $\tau = h\sigma$. The condition $h^{-1}gh \in G_\sigma$ means $gh\sigma = h\sigma$, which is exactly $g\tau = \tau$, meaning $\tau \in \Sigma^g$. Furthermore, the conjugation by $h$ intertwines the representation on $V/V_\sigma$ with $V/V_\tau$, so $\operatorname{det}(tI - \psi_\sigma(h^{-1}gh)_{V/V_\sigma}) = \operatorname{det}(tI - \psi_\sigma(g)_{V/V_\tau})$. Hence the double sum in \eqref{eq-double} re-indexes over all $h \in G$ and all $\tau \in \Sigma^g$:
$$\operatorname{R}_{\overline{\mathcal{A}(\Sigma)}}(t)(g) = \frac{1}{|G|} \sum_{h \in G} \sum_{\tau \in \Sigma^g} \operatorname{det}(tI - g_{V/V_\tau}) = \sum_{\tau \in \Sigma^g} \operatorname{det}(tI - g_{V/V_\tau}).$$
This precisely matches \eqref{eq-reverse}, which complete the proof of Theorem \ref{thm-main}.
\end{proof}

\section{Reflection symmetries and Hybrid fans} \label{sec-ref}
We use \cite{Humphreys90} as our main reference in this section.

Let $W$ be a finite group acting on $V$ linearly and preserving a complete simplicial fan $\Sigma$ in $V$. In this section, we are interested in actions generated by reflections and we want to study the induced representations on the Stanley--Reisner ring $\mathcal{A}(\Sigma)$ and the Artinian reduction $\overline{\mathcal{A}(\Sigma)}$. Since the representations are induced by the action of $W$ on $V$, by passing to the image of the action morphism, we can assume without loss of generality that $W$ is a \emph{finite reflection group} in $V$. That is, $W$ is a finite subgroup of $\operatorname{GL(V)}$ generated by reflections of the form 
$$
s_\alpha(v):=v-2 \frac{\langle v, \alpha\rangle}{\langle\alpha, \alpha\rangle} \alpha
$$ for some nonzero $\alpha \in V$. The action of $W$ on $V$ preserves the inner product $\langle-,-\rangle$.

From the finite reflection group $W$, one can construct (not uniquely) a \emph{root system} $\Phi \subset V$ in the sense of \cite[Chapter 1]{Humphreys90}. This means that $\Phi$  is a finite set of non-zero vectors in $V$ (these vectors are called \emph{roots}) such that for all $\alpha \in \Phi$, the following two conditions hold\footnote{The notion of ``root system'' here differs from that in Lie theory: we don't assume that the roots span $V$ nor we need the crystallographic condition.}:
\begin{equation*}
\begin{aligned}
     &(i)\quad\phi \cap \mathbb{R}\langle\alpha\rangle=\{\alpha,-\alpha\}; \\
&(ii)\quad s_\alpha(\Phi)=\Phi. 
\end{aligned}
\end{equation*}

We make a further assumption on $\Phi$ for later convenience: we assume $\Phi$ consists of unit vectors. With this assumption, our group $W$ is the group generated by all 
reflections $s_\alpha$ for $\alpha \in \Phi$ of the form 
\begin{equation} \label{eq-reflection}
s_\alpha(v):=v-2 \langle v, \alpha\rangle \alpha
\end{equation}
and it is known that the reflections $s_\alpha$ ($\alpha \in \Phi$) exhaust all reflections in $W$ \cite[Proposition 1.14]{Humphreys90}. Fix a such root system $\Phi$ for our group $W$.

Any root system $\Phi$ contains a \emph{simple system} $\Delta$ (whose elements are called \emph{simple roots}), which is a vector space basis for the $\mathbb{R}$-span of $\Phi$ in $V$ and every element of $\Phi$ is a real linear combination of elements in $\Delta$ with coefficients all of the same sign. For the root system $\Phi$ of the group $W$, fix a simple system $\Delta=\left\{\alpha_1, \ldots, \alpha_r\right\} \subset \Phi$. Then $W$ is generated by the reflections $s_i:=s_{\alpha_i}$ for $\alpha_i \in \Delta$, called \emph{simple reflections}. The group $W_J$ generated by a subset $J \subset \Delta$ is called a \emph{standard parabolic subgroup} of W with respect to the simple system $\Delta$. Any subgroup of $W$ conjugate to a standard parabolic subgroup is called a \emph{parabolic subgroup} of $W$.

Consider the open convex cone
$$
C:=\{v \in V:\langle v, \alpha\rangle > 0, \alpha \in \Delta\},
$$
known as the \emph{dominant chamber} of $W$ in $V$. The closure 
$$
D=\{v \in V:\langle v, \alpha\rangle \geq 0, \alpha \in \Delta\}
$$ is known to be a \emph{fundamental domain} for the action of $W$ on $V$ (that is, each $W$-orbit intersects $D$ in exactly one point).

Suppose $\Sigma$ is a fan in $V$ which is invariant under the reflection action of $W$ on $V$. We have the following proposition.

\begin{proposition} \label{prop-unique}
Let $W$ be a finite reflection group acting on $V$ preserving a fan $\Sigma$ in $V$, then in every $W$-orbit of cones in $\Sigma$, there is exactly one cone $\sigma$ whose relative interior, denoted by $\operatorname{relint}(\sigma)$, intersects the fundamental domain $D$.
\end{proposition}

\begin{proof}
Let $\mathcal{O}$ be an orbit of cones in the fan $\Sigma$. Pick any cone $\tau \in \mathcal{O}$ and an arbitrary point $v$ in its relative interior, $v \in \operatorname{relint}(\tau)$. By the properties of the fundamental domain, there exists an element $w \in W$ such that $wv \in D$. Since $\Sigma$ is $W$-invariant, the image $\sigma = w\tau$ is also a cone in the fan $\Sigma$. Since the action is linear, $wv \in \operatorname{relint}(w\tau) = \operatorname{relint}(\sigma)$. Therefore, $wv \in \operatorname{relint}(\sigma) \cap D$, proving that every orbit contains at least one cone whose relative interior intersects $D$.

The proof of the uniqueness needs the following lemma, whose geometric intuition is that the $W$-invariance of the fan forces any cone that ``crosses'' a reflection hyperplane to be symmetric with respect to that hyperplane. 

\begin{lemma} \label{lem-intersect}
Suppose $\Sigma$ is a fan in $V$ which is invariant under the reflection action of $W$ on $V$. If $\sigma \in \Sigma$ and there exists a point $x \in \operatorname{relint}(\sigma)$ that lies on a reflecting hyperplane $H_\alpha$, then $\sigma$ is invariant under the reflection $s_\alpha$ (i.e., $s_\alpha \sigma = \sigma$).   
\end{lemma}

\begin{proof}
    Let $x \in \operatorname{relint}(\sigma) \cap H_\alpha$. By the definition of a reflection, $s_\alpha x = x$. Thus, $x \in \operatorname{relint}(\sigma)$ and $x = s_\alpha x \in s_\alpha(\operatorname{relint}(\sigma)) = \operatorname{relint}(s_\alpha \sigma)$. Since $\Sigma$ is $W$-invariant, $s_\alpha \sigma$ is a cone in the fan. By the intersection property of a polyhedral fan, distinct cones can only intersect at their boundaries (their faces). If two cones in a fan have intersecting relative interiors, they must be the same cone. Thus, $s_\alpha \sigma = \sigma$.
\end{proof}

Now we prove the uniqueness of Proposition \ref{prop-unique}. Suppose there are two cones $\sigma, \tau$ in the same orbit such that the relative interior of both intersect $D$. Since they are in the same orbit, $\tau = w\sigma$ for some $w \in W$. We want to show that $\sigma = \tau$.

Let $x \in \operatorname{relint}(\sigma) \cap D$ and $y \in \operatorname{relint}(\tau) \cap D$. Since $y \in \tau = w\sigma$, we have $w^{-1}y \in \operatorname{relint}(\sigma)$. 
Consider the line segment $L$ in $V$ connecting $x$ to $w^{-1}y$. Since $\operatorname{relint}(\sigma)$ is a convex set, the entire segment $L$ is contained in $\operatorname{relint}(\sigma)$. 

Now, consider the path of the point $z(t)$ along $L$ from $x$ to $w^{-1}y$. 
At $t=0$, $z(0) = x \in D$. 
As $t$ increases, if $z(t)$ leaves $D$, it must cross a reflecting hyperplane $H_\alpha$ that forms a wall of $D$.
According to our Lemma \ref{lem-intersect}, if the relative interior of $\sigma$ intersects $H_\alpha$, then $s_\alpha \in W_\sigma$.

The point $w^{-1}y$ is reached by starting at $x \in D$ and potentially crossing several hyperplanes. Each such hyperplane $H_{\alpha}$ corresponds to a reflection $s_{\alpha}$ that is in the stabilizer of $\sigma$. Consequently, the entire subgroup generated by these reflections---which can map $w^{-1}y$ back into $D$---is contained in $W_\sigma$.

Recall that $y$ is the unique point in the orbit $W(w^{-1}y)$ that lies in $D$. There exists an element $g$ (composed of reflections $s_{\alpha}$ across the walls we crossed) such that $g(w^{-1}y) = y$ and $g \in W_\sigma$.
This implies
$$y = g(w^{-1}y) \in g(\operatorname{relint}(\sigma)) = \operatorname{relint}(\sigma).$$

We now have $y \in \operatorname{relint}(\sigma)$ and $y \in \operatorname{relint}(\tau)$. In a fan, this forces $\sigma = \tau$. Thus, the desired cone is unique.
\end{proof}

Now we prove that the setwise stabilizer $W_\sigma$ for any cone is a parabolic subgroup.
\begin{lemma} \label{lem-parabolic}
    Let $\Sigma$ be a fan in $V$ which is invariant under the reflection action of $W$ on $V$, then the setwise stabilizer $W_\sigma=\{w \in W: w \cdot \sigma=\sigma\}$  of $\sigma$, is equal to the stabilizer $W_y$ of some point y in the relative interior of $\sigma$. As a corollary, $W_\sigma$ is a parabolic group of $W$.
\end{lemma}
\begin{proof}
    The proof of the first part is a standard averaging argument. Pick any point $x$ in the relative interior of $\sigma$. By averaging the orbit of $x$ under the action of $W_\sigma$, we can construct a new point $y$: 
    \begin{equation} \label{eq-averaging}
    y = \frac{1}{|W_\sigma|} \sum_{w \in W_\sigma} w(x).
    \end{equation}
    Because a cone is convex and $w\cdot\sigma= \sigma$ for all $w \in W_\sigma$, the point $y$ must also lie in $\operatorname{relint}(\sigma)$. Furthermore, by construction, $y$ is fixed by every element of $W_\sigma$. Therefore, $W_\sigma$ is a subgroup of the stabilizer of $y$, which we denote as $W_y$.
    
Conversely, suppose we have an element $w \in W_y$ (meaning $w(y) = y$). Because the fan $\Sigma$ is $W$-invariant, the element $w$ must map the cone $\sigma$ to some other cone $\sigma'$ within the fan. Since $w$ fixes $y$, the point $y$ must belong to both $\operatorname{relint}(\sigma)$ and $\operatorname{relint}(\sigma)'$, which implies $\sigma=\sigma'$
Therefore, any $w$ that fixes $y$ must also preserve $\sigma$ as a set. This gives us $W_y \subseteq W_\sigma$. This completes the proof of the first part. 

The second part follows because it is well-known that the stabilizer of any point (or more generally, the point stabilizer of any subset of $V$) is generated by the reflections it contains \cite[Theorem 1.12]{Humphreys90}. By acting the point into the fundamental domain $D$, the point stabilizer is conjugate to the stabilizer of a point in $D$, which is a standard parabolic group of $W$ \cite[Exercise 2 of Section 1.12]{Humphreys90}.
Therefore, the setwise stabilizer $W_\sigma$, which equals $W_y$, is a parabolic subgroup of $W$.
\end{proof}

Following ideas of Colin Crowley and Connor Simpson, we study a new fan that we called \emph{hybrid fan}, which encodes information of the $W$-action on $\Sigma$. The following definition is given by Colin Crowley and Connor Simpson\footnote{Personal communication from Connor Simpson in 4 Feb 2026.}.

\begin{definition}[Crowley--Simpson] \label{def-hybrid fan}
    Let $\Sigma$ be a fan in $V$ which is invariant under the reflection action of $W$ on $V$, the hybrid fan $\Sigma_W$ of $\Sigma$ with the reflection group $W$ is defined as follows: \\
(i) The rays of the hybrid fan $\Sigma_W$ consist of the rays $\rho_i$ of $\Sigma$ contained in the fundamental domain $D$ and rays $\tau_j$ spanned by negative simple roots $-\alpha_j$; \\
(ii) A collection of rays 
$\{\rho_i, \tau_j: i \in I, j \in J\}$ spans a cone, denoted by $\sigma_{I,J}$, 
in $\Sigma_W$ if and only if the associated cone $\sigma'_{I,J} := \operatorname{cone}(\{ w(\rho_i) : i \in I, w \in W_J \})$ is in $\Sigma$.
\end{definition}

In the above definition, since $\sigma'_{I,J}$ is spanned by the entire $W_J$-orbit of the rays $\{\rho_i\}_{i \in I}$, the cone $\sigma'_{I,J}$ is manifestly invariant under the action of the parabolic subgroup $W_J$. Note that for a fixed $I$, the associated cone $\sigma'_{I,J}$ may be the same cone $\sigma$ in the original fan $\Sigma$ for different $J \subset \Delta$, but each $W_J$ is always contained in the setwise stabilizer $W_{\sigma}$ of the cone $\sigma \in \Sigma$.

The following proposition justifies the Definition \ref{def-hybrid fan}, whose proof is given in \cite{Simpson2026more}.
\begin{proposition} \label{prop-justify}
    For any fan $\Sigma$ in $V$ which is invariant under the reflection action of $W$ on $V$, the hybrid fan $\Sigma_W$ is actually a fan in $V$.
\end{proposition}

The following proposition is actually the motivation of Definition \ref{def-hybrid fan}.
\begin{proposition} \label{prop-polytopal}
Suppose $P$ is a full-dimensional polytope in $V$ which is invariant under the reflection action of $W$ on $V$. Let $P/W:=P \cap D$ be the quotient polytope, which is a fundamental domain of $P$ under the action of $W$. Then the normal fan $\Sigma_{P/W}$ of the quotient polytope $P/W$ is the same as the hybrid fan $(\Sigma_P)_W$ of the normal fan $\Sigma_P$ of $P$ with the reflection group $W$.
\end{proposition} 

\begin{proof}
Firstly we show the equality of the rays of $\Sigma_{P/W}$ and $(\Sigma_P)_W$.
The rays of a normal fan correspond to the outer normal vectors of the facets (codimension-1 faces) of the polytope. We must identify the facets of the quotient polytope $P/W = P \cap D$. 
The facets of $P \cap D$ arise from two distinct sources. The first class are the facets of $P$ restricted to $D$: these are the facets of the original $W$-invariant polytope $P$ that intersect the interior of the dominant chamber $D$. Because $\Sigma_P$ is the normal fan of $P$, it is not hard to see that the outer normals to these facets are precisely the rays of $\Sigma_P$ that lie inside $D$ (see for example, \cite[Lemma 2.2 and Proposition 2.4]{gong2026invariants}). That is, the rays $\rho_i$ in $(\Sigma_P)_W$. The second class of the facets of $P \cap D$ comes from facets of the dominant chamber $D$: the dominant chamber $D$ is cut out by the hyperplanes defined by the simple roots, characterized by the inequalities $\langle x, \alpha_j \rangle \ge 0$. The facets of $D$ that intersect $P$ form the remaining facets of $P \cap D$. The outer normal vector to a facet defined by $\langle x, \alpha_j \rangle \ge 0$ (pointing outward from $P \cap D$) is $-\alpha_j$. Since $P$ is full-dimensional and is invariant on the reflection action, $P$ intersects with all the reflection hyperplanes. So all negative simple roots span the rays $\tau_j$ in $(\Sigma_P)_W$. Therefore, the complete set of rays generating the normal fan $\Sigma_{P/W}$ is exactly the rays of the hybrid fan $(\Sigma_P)_W$.

Secondly, we show the equality of the cone structure of $\Sigma_{P/W}$ and $(\Sigma_P)_W$. In a normal fan, a collection of rays spans a cone if and only if the intersection of their corresponding facets forms a non-empty face of the polytope. Let $I$ be a subset of indices for the rays $\rho_i$, and $J$ be a subset of indices for the rays $\tau_j$. We want to show that $\{\rho_i, \tau_j : i \in I, j \in J\}$ spans a cone in $\Sigma_{P/W}$ if and only if $\{ w(\rho_i) : i \in I, w \in W_J \}$ spans a cone in $\Sigma_P$. Let $F_{P/W}$ be a potential face of $P/W$ defined by our chosen rays. It is given by the intersection: $$F_{P/W} = \left( \bigcap_{i \in I} F_i \right) \cap D \cap\left( \bigcap_{j \in J} H_j \right),$$ where $F_i$ is the facet of $P$ with outer normal $\rho_i$, and $H_j = \{x \in V : \langle x, \alpha_j \rangle = 0\}$ is the wall of $D$ corresponding to the simple root $\alpha_j$. By \cite[Proposition 1.15]{Humphreys90}, the intersection of the hyperplanes $\bigcap_{j \in J} H_j$ is exactly the fixed-point subspace $V^{W_J}$ of the standard parabolic subgroup $W_J$. Thus, $F_{P/W}$ consists of the points in $P \cap D$ that are fixed by $W_J$ and simultaneously lie on the facets $F_i$ for all $i \in I$. Because $P$ is $W$-invariant, the facets of $P$ are the union of the $W$-orbits of the facets $F_i$ with outer normal $\rho_i \in D$. If a point $x$ lies on $F_i$ and is fixed by $W_J$, then $x$ must also lie on the translated facet $w(F_i)$ for every $w \in W_J$. Now, it is not hard to see that (see for example \cite[Proposition 2.4]{gong2026invariants}) the existence of the non-empty face $F_{P/W}$ in the quotient polytope is equivalent to the existence of a non-empty face $F_P$ in the original polytope $P$, defined by:
    $$F_P = \bigcap_{i \in I, w \in W_J} w(F_i)$$
By the definition of the normal fan $\Sigma_P$, the intersection of facets $\bigcap_{i \in I, w \in W_J} w(F_i)$ forms a non-empty face of $P$ if and only if their corresponding outer normal vectors, which are exactly $\{ w(\rho_i) : i \in I, w \in W_J \}$, span a cone in $\Sigma_P$. This shows that cone condition of $\Sigma_{P/W}$ is exactly the same with that of the hybrid fan $(\Sigma_P)_W$. 

Therefore, the normal fan of the quotient polytope $\Sigma_{P/W}$ is exactly the hybrid fan $(\Sigma_P)_W$.
\end{proof}

The following proposition is the combinatorial translation of ``quotient of a compact toric variety by a finite reflection group is still compact'', if we believe Conjecture \ref{conj-quotient}.
\begin{proposition} \label{prop-completeness}
    For any fan $\Sigma$ in $V$ which is invariant under the reflection action of $W$ on $V$, the hybrid fan $\Sigma_W$ is complete if the original $W$-invariant fan $\Sigma$ is complete.
\end{proposition}
\begin{proof}
To prove that the hybrid fan $\Sigma_W$ is complete, we must show that the union of all cones in $\Sigma_W$ covers the entire Euclidean space $V$. That is, for any arbitrary vector $v \in V$, there exists a cone in $\Sigma_W$ that contains $v$.

The fundamental domain $D$ is a closed convex polyhedral cone. By the Moreau decomposition theorem for closed convex cones \cite[Theorem 3.2.5]{hiriart2013convex}, any vector $v \in V$ can be uniquely decomposed as
$$v = v_D + v_N,$$
where $v_D \in D$ is the unique closest point to $v$ in $D$, $v_N$ belongs to the polar cone $D^\circ$ of $D$, defined as $\{y \in V : \langle y, x \rangle \leq 0 \text{ for all } x \in D\}$, and $\langle v_D, v_N \rangle = 0$.

Because $D$ is defined by the inequalities $\langle x, \alpha_k \rangle \geq 0$ for all $\alpha_k \in \Delta$, its polar cone $D^\circ$ is precisely the cone generated by the negative simple roots:
$$D^\circ = \operatorname{cone}(\{-\alpha_k : \alpha_k \text{ is a simple root}\}).$$

Therefore, we can write $v_N = \sum_{k} d_k (-\alpha_k)$ where $d_k \geq 0$. 
Let $J$ be the subset of indices where the coefficients are strictly positive, that is, 
$$v_N = \sum_{j \in J} d_j \tau_j, \quad d_j > 0.$$

Apply the orthogonality condition $\langle v_D, v_N \rangle = 0$, we have
$$\left\langle v_D, \sum_{j \in J} d_j (-\alpha_j) \right\rangle = \sum_{j \in J} -d_j \langle v_D, \alpha_j \rangle = 0.$$

Since $v_D \in D$, we know $\langle v_D, \alpha_j \rangle \geq 0$ for all $\alpha_j\in \Delta$. Because $d_j > 0$ for $j \in J$, the sum can only be zero if each individual term is zero:
$$\langle v_D, \alpha_j \rangle = 0 \quad \text{for all } j \in J.$$

This implies that $v_D$ lies on the reflecting hyperplane $H_{\alpha_j}$ for every $j \in J$. Consequently, $v_D$ is invariant under the simple reflections $s_{\alpha_j}$ for $j \in J$, meaning $v_D$ is fixed by the entire parabolic subgroup $W_J$:
$$w(v_D) = v_D \quad \text{for all } w \in W_J.$$

Since we assume the original $W$-invariant fan $\Sigma$ is complete, the point $v_D$ must lie in the relative interior of a unique cone $\sigma \in \Sigma$.

Because $v_D \in \operatorname{relint}(\sigma)$ and $v_D$ lies on the reflecting hyperplanes for $W_J$, by Lemma \ref{lem-intersect}, the entire cone $\sigma$ is invariant under $W_J$. That is, $w(\sigma) = \sigma$ for all $w \in W_J$.

Let $\{\rho_i : i \in I\}$ be the set of rays of $\sigma$ that are contained in the fundamental domain $D$. Because $v_D \in \sigma \cap D$, $v_D$ can be expressed as a non-negative linear combination of these rays:
$$v_D = \sum_{i \in I} c_i \rho_i, \quad c_i \geq 0.$$

Now consider the cone generated by the $W_J$-orbit of these rays:
$$\sigma'_{I,J} = \operatorname{cone}(\{ w(\rho_i) : i \in I, w \in W_J \}).$$

By definition, $\sigma'_{I,J} \subseteq \sigma$. However, because $v_D$ lies in the relative interior of $\sigma$ and $v_D \in \operatorname{cone}(\{\rho_i : i \in I\}) \subseteq \sigma_{I,J}$, the cone $\sigma'_{I,J}$ must have the same dimension as $\sigma$. Since both are convex cones and one is a subset of the other sharing a relative interior point, they must be equal:
$$\sigma'_{I,J} = \sigma.$$

Since $\sigma \in \Sigma$, by Definition \ref{def-hybrid fan} of the hybrid fan, 
$$\sigma_{I,J} = \operatorname{cone}(\{\rho_i : i \in I\} \cup \{\tau_j : j \in J\}) \in \Sigma_W.$$

We now have the decomposition for $v$:
$$v = v_D + v_N = \sum_{i \in I} c_i \rho_i + \sum_{j \in J} d_j \tau_j.$$

Since $c_i \geq 0$ and $d_j > 0$, $v$ is a non-negative linear combination of the generators of $\sigma_{I,J}$. 
Therefore, $v \in \sigma_{I,J}$. Since $v \in V$ can be chosen arbitrary, every vector in the Euclidean space is contained within at least one cone of the hybrid fan $\Sigma_W$. 
Thus, the hybrid fan $\Sigma_W$ is complete.
\end{proof}

The following proposition is the combinatorial translation of ``quotient of a toric orbifold by a finite reflection group is still an orbifold'', if we believe Conjecture \ref{conj-quotient}.
\begin{proposition} \label{prop-simplicial}
If the original $W$-invariant fan $\Sigma$ is simplicial, then the hybrid fan $\Sigma_W$ is also simplicial.
\end{proposition}

\begin{proof}
If $\Sigma$ is simplicial, every cone within it is generated by a linearly independent set of rays. 
Let $\sigma_{I, J}$ be a cone in the hybrid fan $\Sigma_W$ generated by the rays $\{\rho_i\}_{i \in I}$ contained in the dominant chamber $D$, alongside $\{\tau_j\}_{j \in J}$ spanned by the negative simple roots.

By the definition of the hybrid fan, this means that the corresponding cone $\sigma_{I, J}^{\prime}=\operatorname{cone}\left(\left\{w\left(\rho_i\right): i \in I, w \in W_J\right\}\right)$ is in the original fan $\Sigma$. Denote this cone as $\sigma$ for simplicity. Because $\Sigma$ is simplicial, the rays in the orbit $W_J \cdot \{\rho_i\}_{i \in I}$ must form a linearly independent set. 

When the parabolic subgroup $W_J$ acts on this linearly independent set of rays, by \eqref{eq-reflection}, each simple reflection $s_j$ (for $j \in J$) must fix all the rays $\{\rho_i\}_{i \in I}$ or act as a transposition on exactly one pair of rays. Let $J_{\text {triv }}$ be the subset of indices where $s_j$ acts trivially, and $J_{\text {act }}$ be the subset where $s_j$ acts via transposition. Note that $J=J_{\text {triv }} \cup J_{\text {act }}$. 

Consider the action of the parabolic subgroup $W_J$ on the subspace $U := \text{span}(\sigma)$.  The set $J_{\text {triv }}$ corresponds to the root $\alpha_j$ which is orthogonal to $U$, and therefore $\alpha_j \notin U$. For those $\alpha_j$ ($j \in J_{\text {act }}$), because $s_j$ swaps exactly one pair of rays $\rho, \rho'$ and fixes all others, we have 
$\rho' = s_j(\rho) = \rho - c \alpha_j$ for some constant $c \neq 0$, which implies $\alpha_j \in U$. This means that roots in $J_{\text {triv }}$ and roots in $J_{\text {act }}$ are orthogonal hence we have $W_J\cong W_{J_{\text{triv }}}\times W_{J_{\text {act}}}$.

Because the orbit $W_J \cdot \left\{\rho_i: i \in I\right\}$ is generated by starting with the rays $\left\{\rho_i: i \in I\right\}$ and repeatedly applying the transpositions from $J_{\text{act}}$, every index $j \in J_{\text{act}}$ introduces exactly one new linearly independent (since $\sigma$ is simplicial and simple roots are linearly independent) ray into the orbit. The indices in $J_{\text{triv}}$ add no new rays.
Therefore, the total number of rays in the orbit is exactly
$|I| + |J_{\text{act}}|$ and $\dim(U) = |I| + |J_{\text{act}}|$.

We now consider the span of the hybrid ray generators $H:=\{\rho_i\}_{i\in I} \cup \{-\alpha_j\}_{j \in J}$. 
We can split this into the active and trivial roots:
$$\text{span}(H) = \text{span} \left( \{\rho_i\}_{i\in I} \cup \{-\alpha_j\}_{j \in J_{\text{act}}} \right) \oplus \text{span} \left( \{-\alpha_j\}_{j \in J_{\text{triv}}} \right)$$
Recall that for every $j \in J_{\text{act}}$, the root $\alpha_j$ is contained in $U$. Because the orbit $W_J \cdot \left\{\rho_i: i \in I\right\}$ is constructed entirely from $\{\rho_i\}_{i\in I}$ and the active roots, we have an equality of vector spaces:
$$\text{span} \left( \{\rho_i\}_{i\in I} \cup \{-\alpha_j\}_{j \in J_{\text{act}}} \right) = U,$$
whose dimension is $|I| + |J_{\text{act}}|$.

Recall that for every $j \in J_{\text{triv}}$, the root $\alpha_j$ is orthogonal to $U$. Furthermore, the simple roots are themselves linearly independent. Therefore, adding the trivial roots increases the dimension of the span by exactly $|J_{\text{triv}}|$. Thus, the total dimension of the span of $H$ is
$$\dim(\text{span}(H)) = (|I| + |J_{\text{act}}|) + |J_{\text{triv}}| = |I| + |J|.$$

This implies that $H=\{\rho_i\}_{i\in I} \cup \{-\alpha_j\}_{j \in J}$ is linearly independent. Since this holds for any arbitrary cone $\sigma_{I, J}$, the hybrid fan $\Sigma_W$ is simplicial. 
\end{proof}

However, the converse of the above Proposition \ref{prop-simplicial} is false---the hybrid fan can be simplicial even if the original fan is non-simplicial.

\begin{example}
When a reflection group acts on a fan, the orbit of a set of rays $W_J \cdot \{\rho_i\}$ can easily contain more rays than the dimension of the space. The hybrid fan effectively ``compresses'' or ``folds'' this complex orbit into just the $|I| + |J|$ generators. Thus, $\Sigma_W$ can have an exactly correct number of linearly independent generators while representing a cone in $\Sigma$ with far too many rays. We can visualize this geometrically using quotient polytopes. Let the reflection group $W = \mathbb{Z}_2^3$ act on $\mathbb{R}^3$ by sign changes. The closed dominant chamber $D$ is the positive octant, and the negative simple roots are the negative coordinate vectors. Let $\Sigma$ be the normal fan of the regular octahedron.
Because all the vertices (the vertices on the coordinate axes) of the octahedron are meeting by 4 facets, its normal fan contains 3-dimensional cones generated by 4 rays. Therefore, $\Sigma$ is not simplicial. However, by Proposition \ref{prop-polytopal}, the hybrid fan $\Sigma_W$ is the normal fan of the quotient polytope $P/W$, which is the intersection of the octahedron with the dominant chamber $D$. The intersection of the regular octahedron with the positive octant is the standard 3-simplex (a tetrahedron). A tetrahedron is a simple polytope—exactly 3 faces meet at every vertex. Thus, its normal fan $\Sigma_W$ is simplicial. The hybrid fan resolves the non-simplicial nature of the original fan by folding the $W_J$-orbits out of the picture.
\end{example}

For a finite group $G$, recall that we have the complex representation ring $R(G)$, which is isomorphic to the ring of class function on $G$. If $N=\bigoplus N_i$ is a graded representation of $G$, we set
$$
\operatorname{R}_{N}(t)=\sum\left[N_i\right] t^i \in R(G)[t],
$$
where $\left[N_i\right]$ is the class of $N_i$ in $R(G)$. Given a representation $\operatorname{\psi}: G \rightarrow \operatorname{GL}(U)$, we regard $\operatorname{det}(tI-\operatorname{\psi})$ as a a $\mathbb{C}[t]$-valued class function on $G$ whose value at $g \in G$ is $\operatorname{det}(tI-\operatorname{\psi}(g))$. We need the following lemma.

\begin{lemma} \label{lem-ext}
Suppose $\operatorname{dim}(U)=d$, then we have
$$
\operatorname{det}(tI-\operatorname{\psi} )=\sum_{m \geq 0}^{d}(-1)^{d-m}\left[\bigwedge^{d-m} U\right] t^m \in R(G)[t].
$$
\end{lemma}

\begin{proof}
 For any $g \in G$, if $\operatorname{\psi}(g)$ has eigenvalues $\lambda_1, \ldots, \lambda_d$ on $U$, then $\operatorname{det}(tI-\operatorname{\psi}(g) )=\prod_{j=1}^{d}\left(t-\lambda_j\right)$. The conclusion follows from that the trace of $\operatorname{\psi}(g)$ on the exterior power $\bigwedge^{i} U$ is given by the $i$-th elementary symmetric polynomials in $\lambda_1, \ldots, \lambda_d$.
\end{proof}

We are now in a position to give the 

\begin{proof}[Proof of Theorem \ref{thm-ref}] 
We begin with the graded character \eqref{eq-general} of the Artinian reduction $\overline{\mathcal{A}(\Sigma)}$. The Poincaré polynomial of the $W$-invariants $\overline{\mathcal{A}(\Sigma)}^W$ is obtained by taking the inner product with the trivial representation $1_W$:
$$\mathrm{P}_{\overline{\mathcal{A}(\Sigma)}^W}(t) = \left\langle 1_W, \sum_{\sigma \in \Sigma} \frac{\left|W_{\sigma }\right|}{\left|W\right|} \operatorname{Ind}_{W_{\sigma}}^{W}\left(\operatorname{det}(tI-\psi_\sigma)\right) \right\rangle_W.$$

By Frobenius reciprocity, $\langle 1_W, \operatorname{Ind}_{W_\sigma}^W (\chi) \rangle_W = \langle \operatorname{Res}_{W_\sigma}^W (1_W), \chi \rangle_{W_\sigma} = \langle 1_{W_\sigma}, \chi \rangle_{W_\sigma}$ for any character $\chi$ of $W_\sigma$. Summing over the $W$-orbits $[\sigma] \in \Sigma/W$, we get \begin{equation} \label{eq-PP by inner pro}
\mathrm{P}_{\overline{\mathcal{A}(\Sigma)}^W}(t) = \sum_{[\sigma] \in \Sigma/W} \langle 1_{W_\sigma}, \operatorname{det}(tI - \psi_\sigma) \rangle_{W_\sigma}.
\end{equation}

By Proposition \ref{prop-unique}, let $\Gamma = \{ \sigma \in \Sigma : \operatorname{relint}(\sigma) \cap D \neq \emptyset \}$ serve as a complete set of orbit representatives for $\Sigma/W$. For a fixed representative $\sigma \in \Gamma$, let $x \in \operatorname{relint}(\sigma) \cap D$. By taking averaging of $x$ under $W_\sigma$ as in \eqref{eq-averaging}, we get a new point $y\in \operatorname{relint}(\sigma)$. As in the proof of Lemma \ref{lem-parabolic}, $W_\sigma=W_y$. Using a similar argument on the line segment connecting $x$ and $y$ as in the proof of the uniqueness in Proposition \ref{prop-unique}, one can see $y \in D$. By \cite[Section 1.12, Exercise 2]{Humphreys90}, $W_\sigma$ is a standard parabolic subgroup $W_K$ for some set of simple roots $K \subseteq \Delta$.

As a finite reflection group, $W_K$ decomposes the Euclidean space $V$ orthogonally as $V = V_{W_K} \oplus V^{W_K}$, where $V_{W_K}$ is its standard reflection representation (of dimension $|K|$) and $V^{W_K}$ is the subspace fixed pointwise by $W_K$.

Because $W_K$ stabilizes $\sigma$ setwise, the linear span $V_\sigma$ is a $W_K$-invariant submodule of $V$. Since $V_{W_K}$ and $V^{W_K}$ have no isomorphic simple submodules, by complete reducibility, $V_\sigma$ decomposes as
$$V_\sigma = (V_\sigma \cap V_{W_K}) \oplus (V_\sigma \cap V^{W_K}).$$

Since $\sigma$ is a simplicial cone, the $W_K$-action on the rays of $\sigma$ must preserve their linear independence. As in the proof of Proposition \ref{prop-simplicial}, let $J_0 \subseteq K$ be the subset of simple roots corresponding to those simple reflections swapping exactly one pair of rays of $\sigma$. Then $\operatorname{dim}\left(V_\sigma \cap V_{W_K}\right)=\left|J_0\right|$. The remaining simple roots $K^{\prime}=K \backslash J_0$ generated a group $W_{K^{\prime}}$ acting trivially on $V_\sigma$. The group $W_K$ decompose accordingly to a direct product $W_K\cong W_{K'}\times W_{J_0}$.

This allows us to decompose the quotient $U = V/V_\sigma$:
$$V/V_\sigma \cong \frac{V_{W_K}}{V_\sigma \cap V_{W_K}} \oplus \frac{V^{W_K}}{V_\sigma \cap V^{W_K}} =: V_{refl} \oplus V_{triv},$$
It is not hard to see that the reflection part $V_{refl}$ is isomorphic to the standard reflection representation when viewed as a representation of the parabolic subgroup $W_{K'}$, whose dimension (denoted by $r$) is $|K| - |J_0|$. The trivial part $V_{triv}$ is a purely trivial representation of $W_K$, whose dimension (denoted by $k$) is $\dim V^{W_K} - \dim(V_\sigma \cap V^{W_K})$. 

Let $I$ be the indices of the rays of $\sigma$ that lie in $D$. Because $\sigma$ is generated by these $|I|$ rays along with the $|J_0|$ independent reflections, we have $\dim \sigma = |I| + |J_0|$.
We can now precisely calculate $k$. Since $d = |K| + \dim V^{W_K}$, we have
$$k = (d - |K|) - (\dim \sigma - |J_0|) = (d - |K|) - (|I| + |J_0| - |J_0|) = d - |I| - |K|.$$

Since the quotient $U \cong V_{refl} \oplus V_{triv}$, its exterior algebra decomposes as
\begin{equation}
    \bigwedge^m U \cong \bigoplus_{a+b=m} \left( \bigwedge^a V_{refl} \otimes \bigwedge^b V_{triv} \right).
\end{equation}
When taking $W_K$-invariants, the exterior powers of a reflection representation $\left( \bigwedge^a V_{refl} \right)$ of ${W_{K'}}$ are zero for $a > 0$ and $1$-dimensional for $a = 0$, see for example, \cite[Exercise 3(a), p.127]{Bourbaki-Lie456}. Therefore, apply Lemma \ref{lem-ext} to evaluate the inner product for our representative $\sigma$,  only the trivial part survives:
\begin{equation}
    \begin{aligned}
\langle 1_{W_K}, \operatorname{det}(tI - \psi_\sigma) \rangle_{W_K} =& \langle 1_{W_K}, \sum_{m= 0}^{r+k}(-1)^{m}\left[\bigwedge^{m} U\right] t^{r+k-m} \rangle_{W_K} \\=&\sum_{m=0}^k (-1)^m \dim\left( \bigwedge^m V_{triv} \right) t^{r + k - m}\\     
=& \sum_{m=0}^k (-1)^m \binom{k}{m} t^{r+k-m} = t^r (t-1)^k.
    \end{aligned}
\end{equation}

Substituting $r = |K| - |J_0|$ and $k = d - |I| - |K|$, the contribution of the orbit $[\sigma]$ to the Poincar\'e polynomial of the invariant part $\overline{\mathcal{A}(\Sigma)}^W$ is given by 
\begin{equation} \label{eq-contrib of an orbit}
   t^{|K| - |J_0|} (t-1)^{d - |I| - |K|}.  
\end{equation}

By Proposition \ref{prop-completeness} and Proposition \ref{prop-simplicial}, the hybrid fan $\Sigma_W$ is a complete simplicial fan. By \eqref{eq-h-poly}, the $h$-polynomial of the hybrid fan $\Sigma_W$ is
$$h_{\Sigma_W}(t) = \sum_{\sigma_{I,J} \in \Sigma_W} (t-1)^{d - (|I| + |J|)}.$$
We group this sum over hybrid cone $\sigma_{I,J}$ according to its associated cone $\sigma'_{I,J} \in \Sigma$. 

For our fixed representative $\sigma \in \Gamma$ (which is built from the rays in $D$ with index set $I$ extended by the non-trivial action of $W_{J_0}$), we have $\sigma'_{I, J_0} = \sigma$. 
Which other subsets $J\subset \Delta$ yield $\sigma'_{I,J} = \sigma$? Because $J_0$ consists independent simple reflections, any $J\subset \Delta$ yielding $\sigma'_{I,J} = \sigma$ must contains $J_0$.
Any subset $J$ satisfying $J_0 \subseteq J \subseteq K$ will add simple reflections to $J_0$ from $K' = K \setminus J_0$. Recall that $W_{K'}$ acts trivially on $V_\sigma$, meaning these extra reflections do not generate any new rays. Since $W_{K}$ is the setwise stabilizer of $\sigma$, any $J\subset \Delta$ yielding $\sigma'_{I,J} = \sigma$ should be a subset of $K$. Thus, $\sigma'_{I,J} = \sigma$ if and only if $J_0 \subseteq J \subseteq K$.

Therefore, the total contribution to $h_{\Sigma_W}(t)$ from the hybrid cones $\sigma_{I,J}$ associated to the cone $\sigma\in \Sigma$ is
$$\sum_{J_0 \subseteq J \subseteq K} (t-1)^{d - (|I| + |J|)}.$$
Let $m = |J| - |J_0|$. As $J$ ranges from $J_0$ to $K$, $m$ ranges from $0$ to $|K| - |J_0| (= r)$. Then the above sum becomes
$$\sum_{m=0}^r \binom{r}{m} (t-1)^{d - |I| - (|J_0| + m)} = (t-1)^{d - |I| - |J_0| - r} \sum_{m=0}^r \binom{r}{m} (t-1)^{r-m}.$$
Notice that $d - |I| - |J_0| - r = d - |I| - |J_0| - (|K| - |J_0|) = d - |I| - |K|$. By the binomial theorem, the sum simplifies to
$$(t-1)^{d - |I| - |K|} \big( (t-1) + 1 \big)^r = t^{|K| - |J_0|} (t-1)^{d - |I| - |K|}.$$

This perfectly matches the value \eqref{eq-contrib of an orbit} derived from the graded character formula \eqref{eq-general} for the orbit $[\sigma]$. Summing over all orbits in $\Sigma/W$ establishes the equality $\mathrm{P}_{\overline{\mathcal{A}(\Sigma)}^W}(t) = h_{\Sigma_W}(t)$.
\end{proof}

\section{Questions and future directions} \label{sec-ques}
We end the paper by presenting some open problems and directions for future research.

As explained in the introduction, Theorem \ref{thm-ref} gives an evidence of the existence of an isomorphism on the ``cohomology'' level, we leave the following problem to the interested reader.

\begin{problem}
For a finite reflection group $W$ acting linearly on $V$ preserving a complete simplicial fan $\Sigma$, construct an explicit graded ring isomorphism from the Artinian reduction of the Stanley--Reisner ring $\overline{\mathcal{A}(\Sigma_W)}$ of the hybrid fan to the $W$-invariants $\overline{\mathcal{A}(\Sigma)}^W$ of the Artinian reduction of the Stanley--Reisner ring of the original fan $\Sigma$.
\end{problem}

The desired isomorphism in the above problem is constructed for the normal fan of a weight polytope (which is always simple) associated with a regular weight under the action of the corresponding finite reflection group $W$ in \cite{gui2025weyl} and generalized to any simple non-degenerate $W$-symmetric polytopes in \cite{gong2026invariants}. Here, the assumption of being non-degenerate is equivalent to the assumption that the $W$-action on the normal fan is proper, see \cite[Lemma 2.2]{gong2026invariants}. We expect that the definition of the hybrid fan can handle the difficulty of the action being non-proper. 

Actually, we conjecture that there exists an isomorphism on the variety level.

\begin{conjecture} \label{conj-quotient}
  Let $W \subset \mathrm{GL}(V)$ be a finite reflection group preserving a lattice $M \subset V$ and an $M$-rational complete simplicial fan $\Sigma \subset V$. Then the quotient $X_\Sigma / W$ of the  associated toric variety $X_\Sigma$ under the induced action of $W$ is isomorphic to the toric variety $X_{\Sigma_W}$ associated with the hybrid fan $\Sigma_W$.
\end{conjecture}

As mentioned in the introduction, by virtue of Proposition \ref{prop-polytopal}, Crowley--Gong--Simpson's result \cite[Corollary 1.3]{crowley2024toric} show that the above Conjecture \ref{conj-quotient} is true for (possibly non-simplicial) polytopal (hence complete) fans. Besides, their result \cite[Theorem 1.2]{crowley2024toric} implies the above conjecture is also true for the fan $[\sigma]$ (which is non-complete) consist of the faces of a single strongly convex $M$-rational polyhedral cone $\sigma$ (by taking the saturated semigroup $S$ in \cite[Theorem 1.2']{crowley2024toric} to be the intersection of the dual cone $\sigma^\vee$ with $M$). So it is natural to ask the following question.\footnote{The author is informed that Conjecture \ref{conj-quotient} and 
Question \ref{question-quotient} will be addressed in forthcoming work of Connor Simpson.}

\begin{question} \label{question-quotient}
    Is the above Conjecture \ref{conj-quotient} true for any $W$-invariant $M$-rational fan? That is, is the quotient of any normal toric variety by a finite reflection group still toric, and isomorphic to the toric variety associated to the hybrid fan?
\end{question}

As mentioned in the introduction, the Betti numbers and the usual cohomology of a general toric variety are not invariants of the combinatorics of the associated fan in general \cite{mcconnell1989rational}. A better substitute of the usual cohomology for a general toric variety is the intersection cohomology of Goresky--MacPherson\cite{goresky1980intersection,goresky1983intersection}. Stanley\cite{stanley1987generalized} introduced the toric h-polynomial for any convex polytope $P$, which generalize the h-polynomial defined for simplicial polytopes. Actually, Stanley defined this polynomial for any \emph{Eulerian poset}, which including the face poset $\mathcal{F}(\Sigma)$ of any fan $\Sigma$. It is known that Stanley's toric h-polynomial compute the intersection cohomology Poincar\'e polynomial of the associated toric variety when $P$ is rational \cite{denef1991weights,fieseler1991rational}. Barthel--Brasselet--Fieseler--Kaup\cite{barthel2002combinatorial}, Bressler--Lunts\cite{bressler2003intersection}, Karu\cite{karu2004hard},and  Braden\cite{braden2006remarks} developed a theory of combinatorial intersection cohomology of fans (possibly non-rational, which do not define a toric variety). A beautiful result of Barthel--Brasselet--Fieseler--Kaup implies that for any complete fan $\Sigma$, the Poincar\'e polynomial of the combinatorial intersection cohomology of $\Sigma$ is given by Stanley's toric h-polynomial of $\mathcal{F}(\Sigma)$, see \cite[Theorem 3.2]{braden2006remarks}. Consider finite group actions on fans, we have the following 

\begin{question} \label{ques--GCF for CIH}
    Can one derive a graded character formula for the action of a finite group $G$ on the combinatorial intersection group of a complete fan $\Sigma$ in $V$, which is induced by a linear action of $G$ on $V$ preserving $\Sigma$.
\end{question}

The approach of Dolgachev--Lunts\cite{dolgachev1994character} for the character formula of a Weyl group acting on the cohomology of the toric variety associated with Weyl chambers might be useful to the above question. We expect that the above formula should be given by an equivariant version of Stanley's toric h-polynomial, the \emph{equivariant Kazhdan--Lusztig--Stanley polynomial} introduced by Proudfoot \cite{proudfoot2021equivariant}, with the \emph{equivariant p-kernel} given by Stapledon in \cite[Example 1.7]{stapledon2025subdivisions} for the Eulerian poset $\mathcal{F}(\Sigma)$.  

If one obtain a graded character formula for Question \ref{ques--GCF for CIH}, one may wonder whether a similar statement in Theorem \ref{thm-ref} is true for any complete fan $\Sigma$ by replacing $\overline{\mathcal{A}(\Sigma)}$ to the combinatorial intersection cohomology of $\Sigma$, in the spirit of Conjecture \ref{conj-quotient}. However, even if Conjecture \ref{conj-quotient} holds for any complete fans, we do not expect such a simple ``$W$-invariant of the combinatorial intersection cohomology of $\Sigma$ equals the combinatorial intersection cohomology of the hybrid fan $\Sigma_W$" holds. Because even when the involved fans are
rational, computing the intersection cohomology of the quotient need to use the \emph{Decompostion theorem} of Beilinson--Bernstein--Deligne--Gabber\cite{beilinson2018faisceaux} applying to the quotient map to analyze the possible contributions from the singularities of the quotient. We end with 

\begin{problem}
  For a finite reflection group $W$ acting linearly on $V$ preserving a complete fan $\Sigma$, derive a formula which relates  the $W$-invariants of the combinatorial intersection cohomology of $\Sigma$ and the combinatorial intersection cohomology of the hybrid fan $\Sigma_W$.
\end{problem}
\bibliographystyle{amsplain}
\bibliography{toric-symm}

\end{document}